\documentclass{amsart} 
\usepackage{amsmath}
\usepackage{amsthm}
\usepackage{amsfonts}
\usepackage{amssymb}
\usepackage[nocompress,noadjust]{cite}
\usepackage{enumitem}
\usepackage{hyperref}
\usepackage{graphicx}
\usepackage{subcaption}

\graphicspath{ {./figs/} }

\usepackage{todonotes}

\theoremstyle{plain}
  \newtheorem{theorem}{Theorem}

  \newtheorem{proposition}[theorem]{Proposition}
  
  \newtheorem{lemma}[theorem]{Lemma}
  \newtheorem{corollary}[theorem]{Corollary}

\theoremstyle{definition}

  \newtheorem{remark}[theorem]{Remark}

\newcommand{\eps}{\varepsilon}
\newcommand{\C}{\mathbb{C}}
\newcommand{\R}{\mathbb{R}}
\newcommand{\T}{\mathbb{T}}
\renewcommand{\d}{\, d}
\newcommand{\Prob}{\mathbb{P}}
\newcommand{\E}{\mathbb{E}}
\newcommand{\eqd}{\overset{\mathrm{d}}{=}}

\newcommand{\monicp}{\widehat{p_n^{(k_n)}}}
\DeclareMathOperator{\Var}{Var}

\title[Repeated differentiation of random polynomials]{Repeated differentiation of random polynomials with i.i.d. rotationally invariant roots}

\author[S. O'Rourke]{Sean O'Rourke}
\address{Department of Mathematics\\ University of Colorado\\ Campus Box 395\\ Boulder, CO 80309-0395\\USA}
\email{sean.d.orourke@colorado.edu}

\date{\today} 
\subjclass[2020]{Primary 30C15; Secondary 60B10, 31A99}

\begin{document}

\begin{abstract}
Let $p_n$ be a random polynomial of degree $n$ whose roots are independent and identically distributed according to a rotationally invariant probability measure $\mu_0$ on the complex plane with finite logarithmic moment. If $k_n/n\to t\in(0,1)$, we prove that, as $n \to \infty$, the empirical zero measure of the $k_n$-th derivative of $p_n$ converges weakly in probability to a deterministic rotationally invariant probability measure. We describe the limiting measure explicitly in terms of the radial quantile function of $\mu_0$. This proves a conjecture of Hoskins and Kabluchko [\textit{Exp. Math.} 32 (2023), no.~4].
\end{abstract}

\maketitle
    
\section{Introduction}	

The relationship between the zeros of a polynomial and the zeros of its derivatives is classical.
For instance, the Gauss--Lucas theorem implies that every critical point of a polynomial lies in the convex hull formed from its zeros \cite{MR1954841}. 
It gives much less information, though, about the distribution of the critical points  inside the convex hull. 

In this paper, we study the distribution of zeros of random polynomials under repeated differentiation.  
In particular, we consider polynomials with independent and identically distributed (i.i.d.) rotationally invariant roots.
Let $\mu_0$ be a rotationally invariant probability measure on $\mathbb{C}$ that satisfies the following finite logarithmic moment condition:
\begin{equation} \label{eq:lm}
    \int_{\mathbb{C}} \log (1 + |z|) \,d\mu_0(z) < \infty. 
\end{equation}
Let $X_1, X_2, \ldots$ be independent random variables with common distribution $\mu_0$, and define the monic polynomial of degree $n$  
\begin{equation} \label{def:pn}
    p_n(z) = \prod_{j=1}^n(z - X_j), \quad z \in \C. 
\end{equation}
We let $\mu_n$ be the empirical root measure of $p_n$ given by 
\[ \mu_n = \frac{1}{n} \sum_{j=1}^n \delta_{X_j}, \]
where $\delta_z$ is a point mass at $z \in \mathbb{C}$.
Similarly, for any $0 \leq k \leq n-1$, we let $\mu_n^{(k)}$ be the empirical root measure of the $k$-th derivative $p_n^{(k)}$, i.e., 
\[ \mu_n^{(k)} = \frac{1}{n-k} \sum_{j=1}^{n-k} \delta_{W^{(k)}_{n, j}}, \]
where $W^{(k)}_{n,1}, \ldots, W^{(k)}_{n, n-k}$ are the roots of $p_n^{(k)}$ (counted with multiplicity). 
For completeness, when $k \geq n$, we take $\mu_n^{(k)} = \delta_0$. 

Our first main result establishes a deterministic limit for the distribution of the zeros when the number of derivatives is proportional to the degree. 

\begin{theorem} \label{thm:main}
    Let $\mu_0$ be a rotationally invariant probability measure on $\mathbb{C}$ that satisfies \eqref{eq:lm}. 
    Let $X_1, X_2, \ldots$ be independent random variables with common distribution $\mu_0$, and define $p_n$ as in \eqref{def:pn}.
    For any sequence $(k_n)$ of nonnegative integers with 
    \begin{equation} \label{eq:knlim}
        \frac{k_n}{n} \longrightarrow t \in (0, 1)    
    \end{equation}
    as $n \to \infty$, there exists a deterministic probability measure $\mu_t$, depending only on $\mu_0$ and $t$, so that $\mu_{n}^{(k_n)}$ converges weakly in probability to $\mu_t$ as $n \to \infty$.
\end{theorem}

Our second main result gives an explicit description of the limiting distribution $\mu_t$. 
For $r \geq 0$, let 
\begin{equation} \label{def:F0}
    F_0(r) = \mu_0(\{z \in \mathbb{C} : |z| \leq r \})
\end{equation}
be the radial cumulative distribution function of $\mu_0$, and let 
\[ Q_0(u) = \inf \{ r \geq 0 : F_0(r) \geq u \}, \quad 0 < u < 1 \]
be its quantile function (generalized inverse). 

\begin{theorem}[Description of the limit] \label{thm:description}
    For any $0 < t < 1$, the limiting deterministic probability measure $\mu_t$ in Theorem \ref{thm:main} is rotationally invariant, and its radial cumulative distribution function
    \begin{equation} \label{def:Ft}
        F_t(r) = \mu_t( \{ z \in \mathbb{C} : |z| \leq r \})
    \end{equation}
    has quantile function
    \begin{equation} \label{def:qt}
            Q_t(u) = Q_0(t + (1-t) u) \frac{(1-t)u}{t + (1-t)u}, \quad 0 < u < 1.
    \end{equation}
    Equivalently, if $\xi_t$ is a random variable uniformly distributed on $(t, 1)$ and $U$ is an independent random variable uniformly distributed on the unit circle centered at the origin in the complex plane, then $\mu_t$ is the distribution of the random variable 
    \begin{equation} \label{eq:mut}
        Q_0(\xi_t) \frac{\xi_t - t}{\xi_t} U.         
    \end{equation}
\end{theorem}

\begin{remark}
    While finalizing this manuscript, we became aware that Jalowy was independently working on the same problem \cite{jalowy}. 
    We obtained our results independently, without knowledge of each other's work, and we have coordinated to release our preprints simultaneously on the arXiv.
    Jalowy proves results similar to Theorems \ref{thm:main} and \ref{thm:description} when $\mu_0$ is compactly supported and either a rotational invariance condition holds or a certain uniqueness criterion is satisfied.  
    In particular, the methods and results of Jalowy also apply to the case of the heat flow of random polynomials.
\end{remark}

Before presenting some examples, we make a few remarks concerning the main results. 
Both Theorems \ref{thm:main} and \ref{thm:description} can be extended to the case $t = 0$ with only minor modifications to the proof. 
The recent result of Zhu \cite{zhu2026convergencesmallorderderivativesrandom} handles the $t=0$ regime without requiring any assumptions on $\mu_0$, so we have not pursued this case here.
The regime $t = 1$ is also interesting, where one can consider a growing number of remaining roots (e.g., $k_n = n - \log n$) or only finitely many remaining roots (e.g., $k_n = n - 10$). 
While some of the methods used in this work could be applied to the former case, the latter likely requires new methods. 

The logarithmic moment condition \eqref{eq:lm} appears naturally since the proof is based on logarithmic potential theory (see Section \ref{sec:overview} for an overview of the proof).
It remains an open question whether this condition can be relaxed or removed completely. 
Without the rotational invariance condition, we still expect a version of Theorem \ref{thm:main} to hold, but the description in Theorem \ref{thm:description} should no longer be true in general. 

When $\mu_0$ has a point mass of weight $p \in [0, 1]$ at the origin, Theorem \ref{thm:description} implies that the limiting distribution $\mu_t$ will have a point mass of weight $\frac{\max\{p - t, 0\}}{1-t}$ at the origin, as expected. 

It is also natural to ask whether Theorems \ref{thm:main} and \ref{thm:description} can be extended to the situation when the roots are no longer independent or identically distributed. 
One interesting case is that of the characteristic polynomial of random matrices. 
However, the proof method used in this work requires both independence and identical distribution of the roots.

Let us now consider a few examples.  
We start with the case when the radial distribution of $\mu_0$ is discrete, in fact, just a single point mass.  
If $\mu_0$ is uniform on the unit circle (so that its radial distribution is just a point mass at $1$), then $Q_0(u)=1$ and
\[ Q_t(u) = \frac{(1-t)u}{t+(1-t)u}, \quad 0 < u < 1.\]
A numerical simulation of this example (when $t = 1/2$) is given in Figure \ref{fig:circle}. 
\begin{figure}
    \centering
    \begin{subfigure}[t]{0.48\textwidth}
        \centering
        \includegraphics[width=\textwidth]{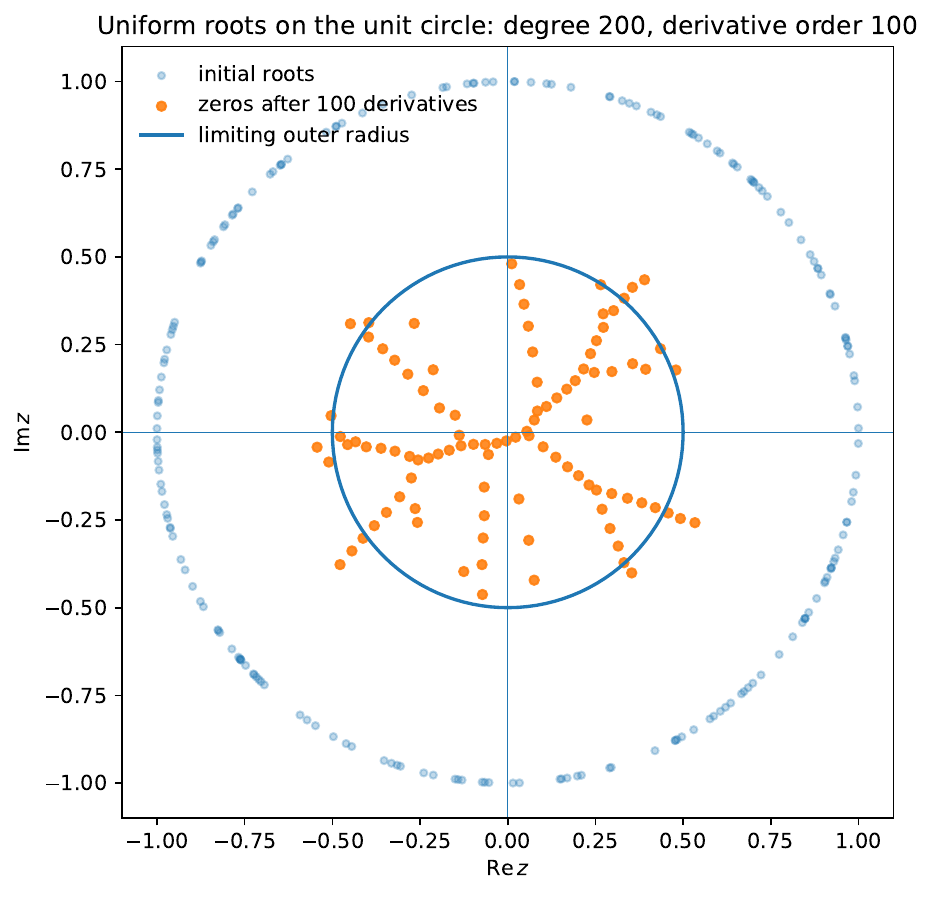}
        \caption{Initial roots and differentiated roots in the complex plane.}
    \end{subfigure}
    \hfill
    \begin{subfigure}[t]{0.48\textwidth}
        \centering
        \includegraphics[width=\textwidth]{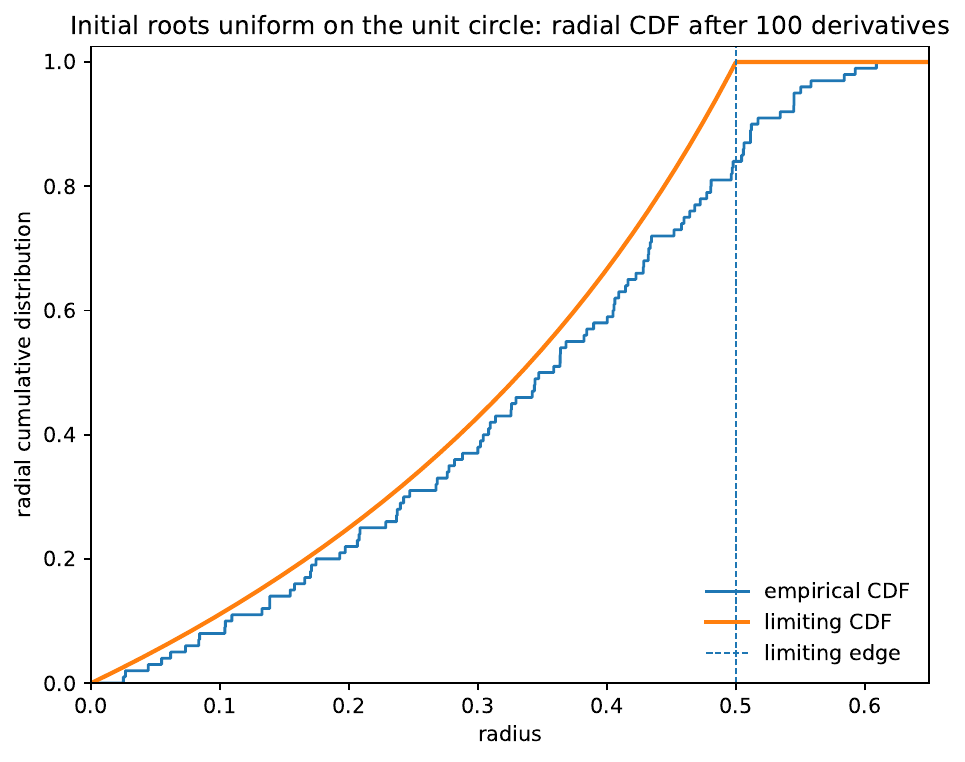}
        \caption{Empirical radial CDF of the differentiated roots compared to the limiting prediction. }
    \end{subfigure}
    \caption{Numerical simulations of Theorems \ref{thm:main} and \ref{thm:description} when $\mu_0$ is the uniform distribution on the unit circle with $n=200$ and $k=100$. }
    \label{fig:circle}
\end{figure}

For the second example, we consider the case when $\mu_0$ is absolutely continuous with respect to the Lebesgue measure on $\C$. 
Take $\mu_0$ to be the uniform distribution on the unit disk.
In this case, $Q_0(u)=\sqrt{u}$ and
\[ Q_t(u) = \frac{(1-t)u}{\sqrt{t+(1-t)u}}, \quad 0 < u < 1. \]
Figure \ref{fig:disk} provides a numerical simulation of this example. 
\begin{figure}
    \centering
    \begin{subfigure}[t]{0.48\textwidth}
        \centering
        \includegraphics[width=\textwidth]{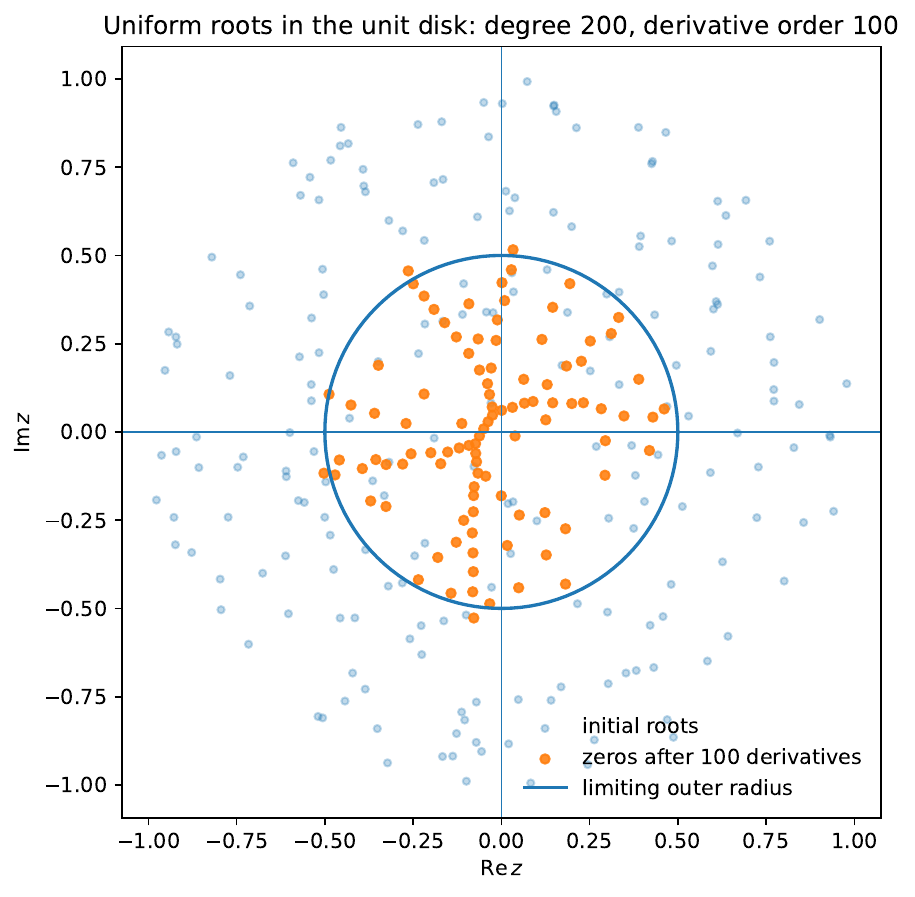}
        \caption{Initial roots and differentiated roots in the complex plane.}
    \end{subfigure}
    \hfill
    \begin{subfigure}[t]{0.48\textwidth}
        \centering
        \includegraphics[width=\textwidth]{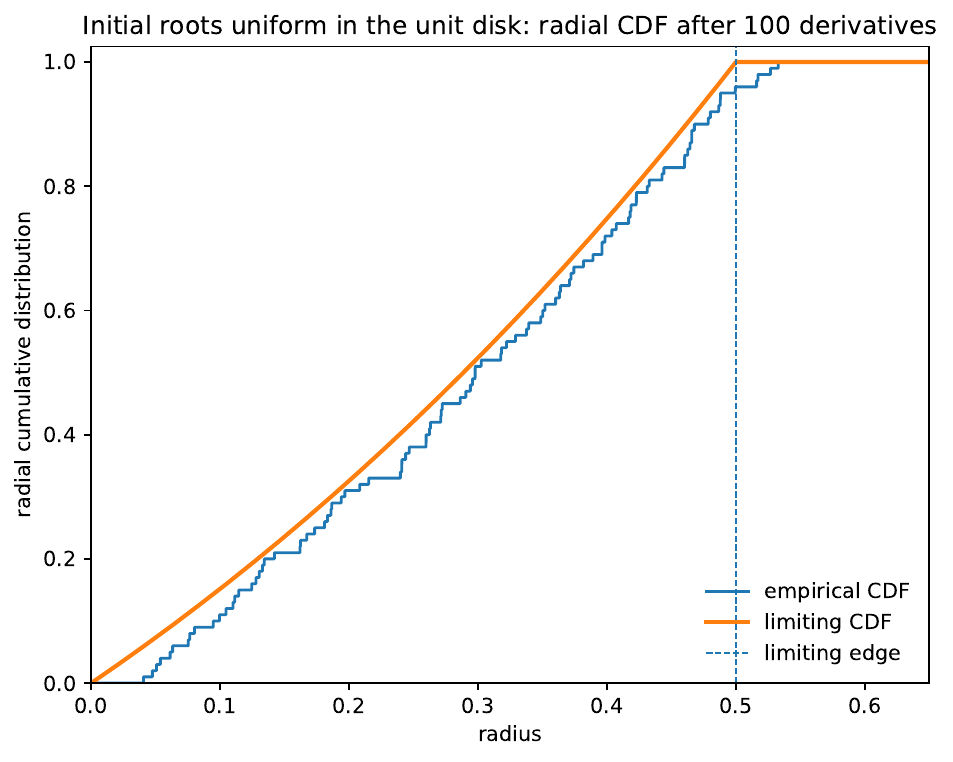}
        \caption{Empirical radial CDF of the differentiated roots compared to the limiting prediction. }
    \end{subfigure}
    \caption{Numerical simulations of Theorems \ref{thm:main} and \ref{thm:description} when $\mu_0$ is the uniform distribution on the unit disk with $n=200$ and $k=100$. }
    \label{fig:disk}
\end{figure}

\subsection{Background and related results}

The study of critical points of polynomials from model \eqref{def:pn} with i.i.d. roots was initiated by Pemantle and Rivin \cite{MR3363974}.  
They conjectured that, for an arbitrary probability measure $\mu_0$ on $\C$, the empirical measure formed from the critical points of $p_n$ converges to $\mu_0$, and proved the conjecture under a finite-energy condition on $\mu_0$.  
Subramanian \cite{MR2970701} studied the conjecture for measures supported on the unit circle, and Kabluchko \cite{MR3283656} proved the conjecture for an arbitrary probability measure on $\C$.  
Since these early works, many results related to the roots of derivatives of this model have appeared. 
In this section, we discuss some of these, with an emphasis on those most related to Theorems \ref{thm:main} and \ref{thm:description}; we have not attempted to cite all works on this model or provide a complete history. 

Extensions of these early works to cases in which the roots need not be independent or identically distributed, such as the case of characteristic polynomials of random matrices, have also been obtained; see \cite{MR3567254,MR3698743,MR3689967} for some specific examples.
Another line of work studies the local relationship between individual roots and nearby critical points.  Under suitable regularity assumptions, roots and critical points can often be paired, and the fluctuations of these pairings can be described; see \cite{MR3940764,MR3896083,MR4136480,MR3689975} and the references therein.  

A fixed number of derivatives was studied by Cheung, Ng, Tsai, and Yam \cite{MR3318313} for roots on the unit circle and by Byun, Lee, and Reddy \cite{MR4474893} in several more general models, including the model with independent roots.  
Almost sure convergence for the first derivative was established by Angst, Malicet, and Poly \cite{MR4711583}, and Michelen and Vu \cite{MR4762159} obtained almost sure convergence for every fixed number of derivatives.

The problem becomes more delicate when the number of derivatives grows with the degree.  
Michelen and Vu \cite{MR4741259} proved that the empirical root measure converges to $\mu_0$ when
\[
    k_n\leq \frac{\log n}{5\log\log n}.
\]
Angst, Nguyen, and Poly \cite{angst2026convergencehigherderivativesrandom} extended this to $k_n=o(n/\log n)$ for a broad class of root distributions.  
The general $k_n = o(n)$ case without any assumptions on the root distribution $\mu_0$ was completed recently by Zhu \cite{zhu2026convergencesmallorderderivativesrandom}. 
These results show that a sub-linear number of derivatives does not change the global root distribution.  

For polynomials with real roots, the linear number of derivatives case was investigated by Steinerberger \cite{MR4011508,MR4669280}.
He derived a nonlocal transport equation for the limiting density and made connections to free probability theory.
A formal proof was given by Hoskins and Kabluchko \cite{MR4669281}. 
In addition, a proof involving finite free probability was presented in \cite{MR4586815}. 
We also refer the reader to \cite{MR4458083} for more analysis of the resulting transport equation.  

The complex and rotationally invariant setting was investigated in \cite{MR4242313}, where a nonlocal radial transport equation was proposed to describe the evolution of roots under repeated differentiation.  
Hoskins and Kabluchko \cite{MR4669281} derived the closed form radial evolution appearing in \eqref{def:qt}. In particular, they conjectured that this formula describes the limiting roots of the $\lfloor tn\rfloor$-th derivative for the model in \eqref{def:pn} with i.i.d. rotationally invariant roots. 
Theorems \ref{thm:main} and \ref{thm:description} prove this conjecture under the  finite logarithmic moment assumption appearing in \eqref{eq:lm}.

Several related results have also been recently obtained for various polynomial models. 
Campbell, Renfrew, and the author \cite{MR4770365} study repeated differentiation for classes of random polynomials with independent coefficients and connect the limiting radial dynamics to free probability.  
Another perspective is offered by Hall, Ho, Jalowy, and Kabluchko \cite{MR5062750} in their study of repeated differentiation and repeated fractional differential operators.   
See also the work of Kabluchko \cite{MR5029824} for the behavior of zeros of trigonometric polynomials under repeated differentiation. 
Galligo, Najnudel, and Vu \cite{galligo2025dynamicsrotationallyinvariantpolynomial}, Najnudel and Vu \cite{najnudel2026rootdynamicsdifferentiatedpolynomials}, and Hall and Perales \cite{hall2026repeateddifferentiationdeterministicpolynomials} proved versions of the evolution for structured root configurations arranged on concentric circles.  
Randomized variants of differentiation have also been studied in \cite{MR4791426,MR4992002}.  

Another extreme regime arises when only finitely many roots remain. Hoskins and Steinerberger \cite{MR4447137} showed that, for independent real roots with suitable moment assumptions, sufficiently high derivatives exhibit Hermite and semicircle universality after appropriately rescaling.  
For extensions, the reader is referred to \cite{arizmendi2025criticalpointsrandompolynomials}. 

We also note some related results concerning the evolution of the zeros under the heat flow for various polynomial models; see, for example, \cite{MR4975505, MR4911809, MR4912666, MR5122040} and references therein.

\subsection{Overview} \label{sec:overview}

The rest of the paper is devoted to the proofs of Theorems \ref{thm:main} and \ref{thm:description}. 
The proofs involve studying the convergence of the logarithmic potential 
\[ L_n^{(k_n)}(z) = \frac{1}{n-k_n} \sum_{j=1}^{n-k_n} \log |z - W_{n,j}^{(k_n)}|, \quad z \in \C, \]
where $W_{n, 1}^{(k_n)}, \ldots, W_{n, n-k_n}^{(k_n)}$ are the roots of $p_n^{(k_n)}$.
We will prove an upper bound for $\limsup_{n \to \infty} L_n^{(k_n)}(z)$ and a lower bound for $\liminf_{n \to \infty} L_n^{(k_n)}(z)$. 
These two bounds take very different forms, and a large part of the argument will be showing that the upper bound and lower bound are actually the same.

We note that the upper bound is obtained almost surely. 
However, the lower bound relies on the one-sided concentration inequality in Lemma \ref{lem:lower-bound}. 
Unfortunately, this result is not strong enough to allow us to show that the lower bound holds almost surely.
A sufficiently strong improvement of this estimate could be used to prove an almost sure weak convergence version of Theorem \ref{thm:main}. 

The paper is organized as follows. 
In Section \ref{sec:prelim}, we will establish the notation, definitions, and tools for the proofs; in particular, this section introduces the necessary logarithmic potential definitions and preliminary results required in the proofs.
We prove Theorems \ref{thm:main} and \ref{thm:description} in Section \ref{sec:proof}.

\subsection*{Acknowledgments}

The author thanks Andrew Campbell, David Renfrew, and Noah Williams for many useful discussions concerning the random polynomial model considered here and Jonas Jalowy for coordinating the releases of the preprints. 

The author has been partially supported by NSF CAREER grant DMS-2143142. 
This work was also supported in part by a grant of access to OpenAI models through the ChatGPT for Academic Researchers program.

\section{Preliminaries} \label{sec:prelim}

This section introduces the necessary notation, definitions, and tools required for the proofs of Theorems \ref{thm:main} and \ref{thm:description}. 

\subsection{Notation}
We denote the imaginary unit by $i$.
The function $\log (\cdot)$ refers to the natural logarithm.  
We define its positive part as
\[ \log_+(x) = \begin{cases}
    \log x & \text{ if } x > 1, \\
    0 & \text{ if } 0 \leq x \leq 1.
\end{cases}\]

For a positive integer $N$, we let $[N] = \{1, \ldots, N\}$ be the discrete interval.  We let $|S|$ denote the cardinality of the finite set $S$.

We use $X \eqd Y$ to denote that the two random variables $X$ and $Y$ have the same distribution. 
We let $\lambda$ denote the Lebesgue measure on $\C$.
For any Borel set $K \subset \C$ and any measurable function $f: K \to \C$, we let $\|f\|_{L^2(K)}$ denote the $L^2$-norm of $f$ defined as
\[ \|f\|_{L^2(K)} = \left( \int_{K} |f(z)|^2 \d \lambda(z) \right)^{1/2}. \]

\subsection{Quantile functions}

Let $Y$ be a real random variable taking only non-negative values.
If 
\[ F_Y(r) = \Prob(Y \leq r), \quad r \geq 0, \]
is the cumulative distribution function (CDF) of $Y$, then the quantile function $Q_Y: [0, 1] \to [0, \infty]$ for (the distribution of) $Y$ is given by 
\begin{equation} \label{def:quantile}
    Q_Y(u) = \inf \{ r \geq 0 : F_Y(r) \geq u\}. 
\end{equation}
Here, we use the convention that $\inf \emptyset = \infty$.
We recall the following standard properties for the quantile function $Q_Y$.
\begin{proposition}[Properties of quantile functions] \label{prop:quantile}
    If $Y$ is a real random variable taking only non-negative values, then the following properties hold. 
    \begin{enumerate}
        \item $Q_Y(0) = 0$.
        \item $Q_Y$ is finite on $[0, 1)$.
        \item $Q_Y$ is non-decreasing on $[0, 1]$ and left-continuous on $(0, 1)$.
        \item $Q_Y(1) = \lim_{u \uparrow 1} Q_Y(u)$, where the limit always exists but may be infinite.
        \item If $\xi_0$ is a random variable uniformly distributed on $(0, 1)$, then 
        \begin{equation*} 
            Y \eqd Q_Y(\xi_0). 
        \end{equation*}
    \end{enumerate}
\end{proposition}
\begin{proof}
    Each property either follows immediately from definition \eqref{def:quantile} or is established in Proposition 1 and Proposition 2 from \cite{MR3072795}; we omit the details. 
\end{proof}

We will often work with quantile functions for the radial distributions of rotationally invariant probability measures on $\C$.  
As in \eqref{def:qt}, we will often only consider these quantile functions on $(0, 1)$ as the values at $0$ and $1$ can be deduced from Proposition \ref{prop:quantile}.

\subsection{Assumptions and preliminary definitions}
We work under the assumptions of Theorem \ref{thm:main}. 
Recall that $\mu_0$ is a measure in the complex plane that satisfies \eqref{eq:lm}, $X_1, X_2, \ldots$ are independent random variables with common distribution $\mu_0$, and $p_n$ is given in \eqref{def:pn}.
We realize each random variable $X_j$ as $X_j = R_j U_j$, where $R_j \geq 0$ is the radial part and $U_j$ is the angular part; 
$R_1, U_1, R_2, U_2, \ldots$ are jointly independent, and for each $j$, $R_j$ has CDF $F_0$ (recall that $F_0$ is defined in \eqref{def:F0}), and $U_j$ is uniformly distributed on the unit circle centered at the origin in the complex plane.
Assumption \eqref{eq:lm} implies that
\begin{equation} \label{eq:l+m}
   \int_{\C} \log_+ |z| \d\mu_0(z) = \E[\log_+ |X_1|] = \E[\log_+ R_1] < \infty,  
\end{equation}
and it will often be more convenient to work with \eqref{eq:l+m} in the proof.

Recall that $(k_n)$ is a sequence satisfying \eqref{eq:knlim}, and $t \in (0, 1)$ will be fixed. 
Throughout the proof, we will assume $n$ is sufficiently large so that $1 \leq k_n \leq n-1$. 
Set $m_n = n - k_n$.

We will let $W^{(k_n)}_{n, 1}, \ldots, W^{(k_n)}_{n, m_n}$ be the roots of $p_n^{(k_n)}$.
It will often be convenient to normalize $p_n^{(k_n)}$ so that it is monic; 
to this end, we define the monic polynomial 
\[ \monicp = \frac{m_n!}{n!} p_{n}^{(k_n)}. \]
Clearly, $\monicp$ has the same roots as $p_n^{(k_n)}$.


\subsection{Logarithmic potential}
We collect the necessary definitions and results from potential theory that we will need in the proof; we refer the reader to \cite{MR4807484, MR1334766} for further details and proofs.

For a finite Borel measure $\nu$ in the complex plane with 
\[ \int_{\C} \log_+|z| \d \nu(z) < \infty, \]
we define its logarithmic potential $L_\nu: \C \to [-\infty, \infty)$ as 
\[ L_\nu(z) = \int_{\C} \log |z - w| \d \nu(w). \]
Here, we interpret $\log 0 = -\infty$.
It follows that 
\begin{equation} \label{eq:log-pot-recovery}
    \Delta L_\nu = 2 \pi \nu
\end{equation}
in the distributional sense, where $\Delta$ is the Laplace operator. 
When $\nu$ is rotationally invariant, we readily see that $L_\nu(z) = L_\nu(|z|)$ for any $z \in \C$.

For $r > 0$ and $z \in \C$, Jensen's formula (see, for instance, \cite{MR4506522}) implies 
\begin{equation} \label{eq:jensen}
    \frac{1}{2\pi}\int_0^{2\pi} \log |z - re^{i \theta}| \d\theta = \log \max\{|z|, r\};
\end{equation}
we will use this identity often to compute the logarithmic potentials that arise in this work.  

$L_n$ will denote the logarithmic potential of $\mu_n$: 
\[ L_n(z) = \frac{1}{n} \sum_{j=1}^n \log |z - X_j|, \quad z \in \mathbb{C}. \]
Similarly, we use $L_n^{(k_n)}$ for the logarithmic potential of $\mu_n^{(k_n)}$:
\[ L_n^{(k_n)}(z) = \frac{1}{m_n}\log |\monicp(z)| =\frac{1}{m_n} \sum_{j=1}^{m_n} \log |z - W^{(k_n)}_{n, j}|, \quad z \in \C. \]

We let $L_0$ be the logarithmic potential of $\mu_0$.
Since $L_0(z) = L_0(|z|)$ by rotational invariance, we will focus on $L_0(x)$ when $x$ is real and positive.  

\begin{proposition}[Properties of $L_0$] \label{prop:L0}
$L_0$ satisfies the following properties:
\begin{enumerate}
    \item For every $x > 0$,
    \[ L_0(x) = \E [ \log \max\{x, R_1\}] = \int_0^1 \log \max \{x, Q_0(u)\} \,du. \]
    \item $L_0(x)$ is finite for every $x > 0$.
    \item $L_0$ is non-decreasing and continuous on $(0, \infty)$.
\end{enumerate}
\end{proposition}
\begin{remark}
    Observe that $L_0$ may not be finite at the origin.  
    For example, if $\mu_0$ has an atom at $0$, then $L_0(0) = -\infty$.
\end{remark}
\begin{proof}[Proof of Proposition \ref{prop:L0}]
By averaging over $U_1$ first and then $R_1$, we have
\begin{equation} \label{eq:L0rep}
    L_0(x) = \E[ \log |x - U_1R_1|] = \frac{1}{2\pi}\E \int_0^{2\pi}\log |x - R_1 e^{i\theta}|  \d \theta = \E [ \log \max\{x, R_1\}],
\end{equation}
where we used \eqref{eq:jensen}.
Since 
\[ \log x \leq \log \max\{x, R_1\} \leq \log_+ x + \log_+ R_1, \]
it follows from \eqref{eq:l+m} that $L_0(x)$ is finite for all $x > 0$. 
In addition, it is easy to see that the expression on the right-hand side of  \eqref{eq:L0rep} is continuous and non-decreasing on $(0, \infty)$. 

If $\xi_0$ is a random variable uniformly distributed on $(0, 1)$, it follows from Proposition \ref{prop:quantile} that $Q_0(\xi_0)$ has the same distribution as $R_1$:
\begin{equation} \label{eq:eq-dist}
    Q_0(\xi_0) \eqd R_1.
\end{equation}
Thus, continuing from \eqref{eq:L0rep}, we conclude that
\[ L_0(x) = \E [ \log \max \{x, Q_0(\xi_0)\}] = \int_0^1 \log \max \{x, Q_0(u)\} \d u \]
for $x > 0$.
\end{proof}

To end this subsection, we note the following deterministic bound.
Recall that $\lambda$ is the Lebesgue measure on $\C$. 
\begin{proposition} \label{prop:l2}
    For any compact set $K \subset \C$, there exists a constant $C > 0$ (depending only on $K$) so that 
    \begin{equation} \label{eq:l2}
        \left( \int_K \log^2 |z - w| \d \lambda(z) \right)^{1/2} \leq C (1 + \log_+|w|) 
    \end{equation}
    for any $w \in \C$. 
\end{proposition}
\begin{proof}
    Choose $r > 1$ so that $K \subset \{z \in \mathbb{C} : |z| < r\}$. 
    If $|w| \leq r + 1$, then $K-w$ lies in the disk $\{z \in \C: |z| < 2r + 2\}$, and so
    \[ \int_K \log^2 |z-w|\d \lambda(z) \leq \int_{|z| < 2r + 2} \log^2 |z|\d \lambda(z) \leq C_r \]
    for a constant $C_r > 0$ depending only on $r$ (and hence only on $K$) since $\log^2 |\cdot|$ is locally integrable. 
    If $|w| > r + 1$, then $1 \leq |z-w| \leq 2|w|$ for any $z \in K$. 
    It follows that 
    \[ \int_K \log^2 |z-w| \d \lambda(z) \leq 4 \lambda(K)\log_+^2 |w|. \]
    The bound in \eqref{eq:l2} follows by choosing $C$ appropriately in terms of the constants $C_r$ and $4 \lambda(K)$. 
\end{proof}

\subsection{The limiting distribution}
Define $\mu_t$ to be the distribution (in the complex plane) of the random variable given in \eqref{eq:mut}.
In particular, $\mu_t$ is clearly rotationally invariant by construction. 
Set
\begin{equation} \label{def:lt}
    L_t(z) = \frac{1}{1-t}\int_t^1 \log \max \left\{|z|, Q_0(u) \frac{u-t}{u}\right\} \d u 
\end{equation}
for $z \in \C$. 

\begin{proposition}[Properties of $\mu_t$] \label{prop:mut}
    The rotationally invariant measure $\mu_t$ on $\C$ satisfies the following properties:
    \begin{enumerate}
    \item One has
    \begin{equation} \label{eq:mutflog}
        \int_\C \log (1 + |z|) \d \mu_t(z) < \infty,
    \end{equation}
    and $L_t$ is the logarithmic potential of $\mu_t$.
    \item For any compact set $K \subset \C$, there exists a constant $C > 0$ (depending only on $K$) so that
    \begin{equation} \label{eq:mutl2}
        \|L_t\|_{L^2(K)} \leq C \left( 1 + \int_\C \log_+ |w| \d \mu_t(w) \right) < \infty. 
    \end{equation}
    \item The radial CDF $F_t$ (given in \eqref{def:Ft}) for $\mu_t$ has quantile function $Q_t$ (defined in \eqref{def:qt}).
    \end{enumerate}
\end{proposition}
\begin{proof}
    Recall that $\mu_t$ is the distribution of the random variable in \eqref{eq:mut}, where $\xi_t$ is uniform on $(t, 1)$ and $U$ is independent and uniformly distributed on the unit circle in the complex plane.  
    Let $\xi_0$ be a random variable uniformly distributed on $(0, 1)$. 
    Since $Q_0(u) \frac{u -t}{u} \leq Q_0(u)$ for $t < u < 1$, we have
    \begin{align*}
        \int_\C \log (1 + |z|)\d \mu_t &= \frac{1}{1-t} \int_t^1 \log \left( 1 + Q_0(u) \frac{u - t}{u}\right) \d u \\
        &\leq \frac{1}{1-t}\int_0^1 \log \left( 1 + Q_0(u) \right) \d u \\
        &= \frac{1}{1-t}\E [ \log (1 + Q_0(\xi_0))] \\
        &= \frac{1}{1-t}\E [ \log (1 + R_1)] \\
        &= \frac{1}{1-t}\int_\C \log (1 + |z|) \d \mu_0(z) \\
        &< \infty,
    \end{align*}
    where we used \eqref{eq:eq-dist} and the last integral is finite by assumption \eqref{eq:lm}.
    This establishes \eqref{eq:mutflog}. 
    For $z \in \C$ with $z \neq 0$, we have 
    \begin{align*}
        L_t(z) &= \E \left[ \log \max \left\{|z|, Q_0(\xi_t) \frac{\xi_t - t}{\xi_t} \right\}  \right] \\
        &= \frac{1}{2\pi} \E \left[ \int_0^{2\pi} \log \left||z| - Q_0(\xi_t) \frac{\xi_t -t}{\xi_t} e^{i \theta}\right| \d \theta \right] \\
        &= \E \left[ \log\left| |z| - Q_0(\xi_t) \frac{\xi_t - t}{\xi_t} U \right| \right] \\
        &= L_{\mu_t}(|z|)
    \end{align*}
    by \eqref{eq:jensen}. 
    Since $\mu_t$ is rotationally invariant, we conclude that $L_t$ is the logarithmic potential of $\mu_t$. 

    To establish the $L^2$-norm bound in \eqref{eq:mutl2}, we will apply Proposition \ref{prop:l2}. 
    Let $K \subset \C$ be compact. 
    Minkowski's integral inequality (see Theorem 202 in \cite{MR944909}) gives
    \begin{align*}
        \|L_t\|_{L^2(K)} &= \left( \int_K \left( \int_\C \log |z - w | \d \mu_t(w) \right)^2 \d \lambda(z) \right)^{1/2} \\
        &\leq \int_\C \left( \int_K \log^2 |z -w | \d \lambda(z) \right)^{1/2} d\mu_t(w).
    \end{align*}
    By bounding the inner integral on the right-hand side using \eqref{eq:l2} from Proposition \ref{prop:l2}, we obtain precisely the bound in \eqref{eq:mutl2} (whose value is finite by \eqref{eq:mutflog}). 
    
    To see that $Q_t$ is the quantile function corresponding to the radial CDF $F_t$, we define
    \begin{equation} \label{def:ft}
        f_t(u) = Q_0(u) \frac{u-t}{u}, \quad t < u < 1.         
    \end{equation}
    By definition, $F_t$ is the CDF of $f_t(\xi_t)$, and 
    \[ \xi_t \eqd t + (1-t) \xi_0. \]
    This means $F_t$ is the CDF of 
    \[ f_t(\xi_t) \eqd f_t(t + (1-t)\xi_0) = Q_t(\xi_0). \]
    The conclusion now follows from Proposition \ref{prop:inverse} below.
\end{proof}

\begin{proposition} \label{prop:inverse}
    Let $\xi_0$ be a random variable uniformly distributed on $(0, 1)$, and let $g: (0, 1) \to [0, \infty)$ be non-decreasing and left-continuous.
    If $Y = g(\xi_0)$ has CDF $F_Y$ and quantile function 
    \[ Q_Y(u) = \inf \{ r \geq 0 : F_Y(r) \geq u \}, \quad 0 < u < 1, \]
    then $Q_Y(u) = g(u)$ for all $0 < u < 1$.
\end{proposition}
\begin{proof}
    Note that $Q_Y(u)$ can be extended to $u = 0$ and $u=1$ by Proposition \ref{prop:quantile}.
    Fix $0 < u < 1$.
    Since $g$ is non-decreasing, we find
    \[ F_Y(g(u)) = \Prob(g(\xi_0) \leq g(u)) \geq \Prob(\xi_0 \leq u) = u, \]
    and so $Q_Y(u) \leq g(u)$.
    For the reverse inequality, take $r < g(u)$.
    By left-continuity, there exists $v < u$ so that $g(v) > r$.
    Since $g$ is non-decreasing, this means $\{0 < s < 1 : g(s) \leq r\} \subset (0, v)$, and hence
    \[ F_Y(r) = \Prob ( g(\xi_0) \leq r) \leq v < u. \]
    In other words, $Q_Y(u) \geq r$.  
    Since $r < g(u)$ was arbitrary, we conclude that $Q_Y(u) \geq g(u)$. 
\end{proof}

\section{Proofs of Theorems \ref{thm:main} and \ref{thm:description}} \label{sec:proof}

This section is devoted to the proofs of Theorems \ref{thm:main} and \ref{thm:description}.
In fact, we focus only on Theorem \ref{thm:main} as we can already handle the proof of Theorem \ref{thm:description}. 
\begin{proof}[Proof of Theorem \ref{thm:description}]
    Recall that $\mu_t$ is defined as the distribution of the complex-valued random variable in \eqref{eq:mut}. 
    The distribution is clearly rotationally invariant.  
    In addition, the description of its radial quantile function in \eqref{def:qt} follows from Proposition \ref{prop:mut}. 
\end{proof}

\subsection{Proof of Theorem \ref{thm:main}}

It remains to establish Theorem \ref{thm:main}.  
To do so, we will establish convergence of the logarithmic potential $L_n^{(k_n)}$ to the limiting logarithmic potential $L_t$. 

\begin{lemma}[Pointwise convergence of the logarithmic potential] \label{lem:pointwise}
    For every $z \in \C$ with $z \neq 0$, 
    \[ L_n^{(k_n)}(z) \longrightarrow L_t(z) \]
    in probability as $n \to \infty$.
\end{lemma}

We prove Lemma \ref{lem:pointwise} in Section \ref{sec:pointwise}. 
Assuming Lemma \ref{lem:pointwise}, we now complete the proof of Theorem \ref{thm:main}. 
We need a few helpful lemmata first.

\begin{lemma} \label{lemma:mahler}
    The roots $W^{(k_n)}_{n, 1}, \ldots, W^{(k_n)}_{n, m_n}$ of $p_n^{(k_n)}$ satisfy the following deterministic bound: 
    \begin{equation} \label{eq:mahler}
        \sum_{j=1}^{m_n} \log_+ |W^{(k_n)}_{n,j}| \leq \sum_{j=1}^n \log_+|X_j|. 
    \end{equation}
\end{lemma}
\begin{proof}
    For a polynomial 
    \[ q(z) = a_d \prod_{j=1}^d (z - \zeta_j) \]
    with roots $\zeta_1, \ldots, \zeta_d \in \C$ and $a_d \in \C$, define its Mahler measure $\mathcal{M}(q)$ as
    \[ \mathcal{M}(q) = |a_d| \prod_{j=1}^d \max \{1, |\zeta_j|\}. \]
    Mahler's inequality \cite{MR133437} implies that 
    \[ \mathcal{M} \left( \frac{1}{d} q' \right) \leq \mathcal{M}(q) \]
    for any degree $d$ polynomial $q$; we also refer the reader to \cite{MR4031380} for a more modern treatment and generalizations of this bound. 
    Iterating Mahler's inequality, we see that
    \[ \mathcal{M}\left( \monicp \right) \leq \mathcal{M}(p_n) \]
    since both $\monicp$ and $p_n$ are monic. Taking logarithms of both sides of this bound yields \eqref{eq:mahler}. 
\end{proof}

As a consequence of Lemma \ref{lemma:mahler}, we obtain the following (deterministic) $L^2$-norm bound.

\begin{corollary} \label{cor:lnl2}
    For any compact set $K \subset \C$, there exists a constant $C > 0$ (depending only on $K$) so that
    \[ \left\| L_n^{(k_n)} \right\|_{L^2(K)} \leq C \left( 1 + \frac{1}{m_n} \sum_{j=1}^n \log_+ |X_j| \right). \]
\end{corollary}
\begin{proof}
    Let $K \subset \C$ be compact, and recall that $\lambda$ is the Lebesgue measure on $\C$.  
    By the triangle inequality for the $L^2$-norm, 
    \begin{align*}
        \|L_n^{(k_n)}\|_{L^2(K)} &= \left( \int_K \left( \frac{1}{m_n}\sum_{j=1}^{m_n} \log |z - W^{(k_n)}_{n,j} | \right)^2 \d \lambda(z) \right)^{1/2} \\
        &\leq \frac{1}{m_n}\sum_{j=1}^{m_n} \left( \int_K \log^2|z - W^{(k_n)}_{n,j}| \d \lambda(z) \right)^{1/2}. 
    \end{align*}
    By Proposition \ref{prop:l2}, there exists a constant $C > 0$ (depending only on $K$) so that 
    \begin{align*}
        \frac{1}{m_n}\sum_{j=1}^{m_n} \left( \int_K \log^2|z - W^{(k_n)}_{n,j}| \d \lambda(z) \right)^{1/2} &\leq \frac{C}{m_n}\sum_{j=1}^{m_n} (1 + \log_+ |W^{(k_n)}_{n,j}|) \\
        &= C \left( 1 + \frac{1}{m_n} \sum_{j=1}^{m_n} \log_+|W^{(k_n)}_{n,j}| \right).
    \end{align*}
    The conclusion now follows from Lemma \ref{lemma:mahler}. 
\end{proof}

We now upgrade the pointwise convergence in Lemma \ref{lem:pointwise} to local $L^1$ convergence. 
Recall that $\lambda$ is the Lebesgue measure on $\C$. 

\begin{lemma} \label{lemma:l1}
    For any compact set $K \subset \C$, 
    \[ \int_K |L_n^{(k_n)}(z) - L_t(z)| \d \lambda(z) \longrightarrow 0 \]
    in probability as $n \to \infty$.
\end{lemma}
\begin{proof}
    Let $K \subset \C$ be compact. 
    Fix $\eps > 0$, and define the random set 
    \[ E_{n} = \{z \in K : |L_n^{(k_n)}(z) - L_t(z)| > \eps \}. \]
    By Fubini's theorem,  
    \[ \E [\lambda(E_n)] = \int_K \Prob( |L_n^{(k_n)}(z) - L_t(z)| > \eps) \d \lambda(z). \]
    Therefore, by Lemma \ref{lem:pointwise} and the dominated convergence theorem, $\E[\lambda(E_n)] \to 0$ as $n \to \infty$, 
    and hence $\lambda(E_n)$ converges to zero in probability. 

    By the Cauchy-Schwarz and triangle inequalities, 
    \begin{align}
        \int_K |L_n^{(k_n)}(z) - L_t(z)| \d \lambda(z) &\leq \eps \lambda(K) + \int_{E_n} |L_n^{(k_n)}(z) - L_t(z)| \d \lambda(z) \nonumber \\
        &\leq \eps \lambda(K) + \|L_n^{(k_n)} - L_t\|_{L^2(K)} \lambda(E_n)^{1/2} \nonumber \\
        &\leq \eps \lambda(K) + \|L_n^{(k_n)}\|_{L^2(K)} \lambda(E_n)^{1/2} + \|L_t\|_{L^2(K)} \lambda(E_n)^{1/2}. \label{eq:finintbnd}
    \end{align}
    The first term on the right-hand side of \eqref{eq:finintbnd} can be made arbitrarily small by choice of $\eps$.  
    The third term converges to zero in probability since $\lambda(E_n)$ converges to zero in probability and $\|L_t\|_{L^2(K)}$ is a finite deterministic constant by \eqref{eq:mutl2}. 
    For the second term, Corollary \ref{cor:lnl2} implies there exists a constant $C > 0$ (depending only on $K$) so that 
    \begin{align*}
        \lambda(E_n)^{1/2} \|L_n^{(k_n)}\|_{L^2(K)} &\leq C \lambda(E_n)^{1/2} \left( 1 + \frac{n}{m_n} \frac{1}{n} \sum_{j=1}^n \log_+|X_j| \right). 
    \end{align*}
     Notice that $\frac{n}{m_n} \to \frac{1}{1-t}$ as $n \to \infty$ by assumption and 
     \[ \frac{1}{n} \sum_{j=1}^n \log_+|X_j| \longrightarrow \E [ \log_+ |X_1|] < \infty \]
     in probability by the law of large numbers (recall \eqref{eq:l+m}).
     We conclude that $\lambda(E_n)^{1/2} \|L_n^{(k_n)}\|_{L^2(K)}$ converges to zero in probability, and the proof is complete.
\end{proof}

We now have all the tools we need to prove Theorem \ref{thm:main}.

\begin{proof}[Proof of Theorem \ref{thm:main}]
    Let $\varphi: \C \to \C$ be a smooth compactly supported function. 
    We see that
    \begin{align*}
        \int_\C \varphi(z) \d \mu_n^{(k_n)}(z) &= \frac{1}{2\pi} \int_\C L_n^{(k_n)}(z) \Delta \varphi(z) \d \lambda(z) \\
        &\longrightarrow \frac{1}{2\pi}\int_\C L_t(z) \Delta \varphi(z) \d \lambda(z) \\
        &= \int_\C \varphi(z) \d \mu_t(z)
    \end{align*}
    in probability as $n \to \infty$, where the first equality holds by \eqref{eq:log-pot-recovery}, the convergence follows from Lemma \ref{lemma:l1}, and the last equality is a consequence of Proposition \ref{prop:mut} and \eqref{eq:log-pot-recovery}. 
    By a standard approximation argument (approximating any continuous compactly supported function on $\C$ by a smooth compactly supported function), we find that $\mu_n^{(k_n)}$ converges vaguely to $\mu_t$ in probability as $n \to \infty$. 
    Since $\mu_t$ is a deterministic probability measure, we can upgrade this vague convergence to weak in probability convergence.
    The argument is standard, and we refer the reader to Theorem 4.19 in \cite{MR3642325} for details.
\end{proof}

\subsection{Proof of Lemma \ref{lem:pointwise}} \label{sec:pointwise}

It remains to establish Lemma \ref{lem:pointwise}. 
To this end, define the polynomial
\begin{equation} \label{def:qn}
    q_n(z) = \frac{1}{k_n!} p_n^{(k_n)}(z). 
\end{equation}
It will be slightly more convenient to work with $q_n$ rather than $p_n^{(k_n)}$. 
We also define the function
\begin{equation} \label{def:vt}
    v_t(r) = \inf_{\rho > 0} \{ L_0(r + \rho) - t \log \rho \}. 
\end{equation}

Lemma \ref{lem:pointwise} will follow from the following two results.

\begin{lemma} \label{lem:qnpointwise}
    For any $r > 0$, 
    \[ \frac{1}{n} \log |q_n(r)| \longrightarrow v_t(r) \]
    in probability as $n \to \infty$. 
\end{lemma}

\begin{lemma} \label{lem:vt}
    For any $r > 0$, 
    \[ v_t(r) = \int_t^1 \log \max\left\{ r, Q_0(u) \frac{u-t}{u} \right\} \d u - t \log t - (1-t)\log(1-t). \]
\end{lemma}

We can now prove Lemma \ref{lem:pointwise} assuming the previous two lemmata.

\begin{proof}[Proof of Lemma \ref{lem:pointwise}]
    Recall that 
    \[ p_n^{(k_n)}(z) = k_n! \sum_{1 \leq i_1 < i_2 < \cdots < i_{k_n} \leq n} \prod_{j \not\in \{i_1, \ldots, i_{k_n}\}} (z - X_j). \]
    From this form for $p_n^{(k_n)}$ and the rotational invariance of $\mu_0$, it follows that 
    \[ |p_n^{(k_n)}(z)| \eqd |p_n^{(k_n)}(|z|)| \] 
    for any $z \in \C$. 
    As a consequence, 
    \[ L_n^{(k_n)}(z) = \frac{1}{m_n} \log \left| \monicp(z) \right| \eqd L_n^{(k_n)}(|z|) \]
    for any $z \in \C$.  
    In addition, by the rotational invariance of $\mu_t$, we have $L_t(z) = L_t(|z|)$ for any $z \in \C$.
    Therefore, in order to prove Lemma \ref{lem:pointwise}, it suffices to prove that for any $r > 0$, 
    \begin{equation} \label{eq:pointwiseshow}
        L_n^{(k_n)}(r) \longrightarrow L_t(r)
    \end{equation}
    in probability as $n \to \infty$. 

    To this end, fix $r > 0$. 
    The leading coefficient of $q_n$ is $\binom{n}{k_n}$ by our choice of normalization, and hence
    \[ \monicp = \binom{n}{k_n}^{-1} q_n. \]
    By Stirling's approximation, 
    \[ \lim_{n \to \infty} \frac{1}{n} \log \binom{n}{k_n}^{-1} = t \log t + (1-t) \log (1-t). \]
    Thus, by Lemma \ref{lem:qnpointwise} 
    \[ \frac{1}{n} \log \left| \monicp(r) \right| \longrightarrow v_t(r) + t \log t + (1-t) \log (1-t) \]
    in probability as $n \to \infty$. 
     By Lemma \ref{lem:vt}, 
     \[ v_t(r) + t \log t + (1-t) \log (1-t) = \int_t^1 \log \max\left\{ r, Q_0(u) \frac{u-t}{u} \right\} \d u = (1-t)L_t(r), \]
     where we recall the definition of $L_t$ given in \eqref{def:lt}. 
     Since $n/m_n \to (1-t)^{-1}$ by assumption \eqref{eq:knlim}, it follows that
     \begin{align*}
         L^{(k_n)}_n(r) = \frac{1}{m_n} \log \left| \monicp(r) \right|
         = \frac{n}{m_n}\frac{1}{n} \log \left| \monicp(r) \right| 
         \longrightarrow L_t(r)
     \end{align*}
     in probability as $n \to \infty$. 
     This establishes \eqref{eq:pointwiseshow}, and the proof is complete. 
\end{proof}

The rest of the paper is devoted to the proofs of Lemmas \ref{lem:qnpointwise} and \ref{lem:vt}.
We prove Lemma \ref{lem:vt} in Section \ref{sec:vt}, and Lemma \ref{lem:qnpointwise} is established in Section \ref{sec:qnpointwise}.

\subsection{Proof of Lemma \ref{lem:vt}} \label{sec:vt}

Recall that $t \in (0, 1)$ is fixed.
Lemma \ref{lem:vt} will follow from the next lemma.
For $\beta \in (t, 1]$ and $r > 0$, we define 
\begin{equation} \label{def:Bt}
     B_t(\beta, r) = \beta \log \beta - (\beta - t) \log (\beta - t) - t \log t + (\beta - t) \log r + \int_\beta^1 \log Q_0(u) \d u, 
\end{equation}
with the convention that when $\beta = 1$, the integral term is zero; 
we can extend the definition to $\beta = t$ by taking a right limit (in the extended-real sense):
\[ B_t(t, r) = \lim_{\beta \downarrow t} B_t(\beta, r) = \int_t^1 \log Q_0(u) \d u. \]

\begin{lemma} \label{lem:vt-bt}
    For every $r > 0$, 
    \begin{equation} \label{eq:vtridenti}
        v_t(r) = \sup_{t \leq \beta \leq 1} B_t(\beta, r). 
    \end{equation}
\end{lemma}
\begin{proof}
    It follows from the definition of $v_t$ given in \eqref{def:vt} that
    \[ v_t(r) = \inf_{\rho > r} \{ L_0(\rho) - t \log (\rho - r) \}. \]
    
    We begin by showing 
    \begin{equation} \label{eq:firstineqvt}
        v_t(r) \geq \sup_{t \leq \beta \leq 1} B_t(\beta, r) 
    \end{equation}
    for every $r > 0$. 
    Fix $\beta \in [t, 1]$ and $\rho > r$. 
    We have 
    \begin{align*}
        L_0(\rho) &= \int_0^1 \log \max \{ \rho, Q_0(u) \} \d u \\
        &\geq \int_0^\beta \log \rho \d u + \int_{\beta}^1 \log Q_0(u) \d u \\
        &= \beta \log \rho + \int_\beta^1 \log Q_0(u) \d u, 
    \end{align*}
    and so 
    \[ L_0(\rho) - t \log (\rho - r) \geq \int_\beta^1\log Q_0(u) \d u + \beta \log \rho - t \log (\rho - r). \]
    Taking the infimum in  $\rho$ yields
    \[ v_t(r) \geq \int_\beta^1 \log Q_0(u) \d u + \inf_{\rho > r} \{ \beta \log \rho - t \log (\rho - r) \}. \]
    A standard minimization argument shows the expression is minimized when 
    \[ \frac{\beta}{\rho} = \frac{t}{\rho - r} \text{ for } \beta > t, \]
    and hence
    \begin{align*}
        \inf_{\rho > r} \{ \beta \log \rho - t \log (\rho - r) \} = (\beta - t) \log r + \beta \log \beta - (\beta - t) \log (\beta - t) - t \log t. 
    \end{align*}
    When $\beta = t$, 
    \[ \inf_{\rho > r} \{ \beta \log \rho - t \log (\rho - r) \} = t \inf_{\rho > r} \log \frac{\rho}{\rho - r} = 0. \]
    In either case, we conclude that 
    \[ v_t(r) \geq B_t(\beta, r) \]
    for every $\beta \in [t, 1]$. 
    Taking a supremum in  $\beta$ gives \eqref{eq:firstineqvt}. 

    We now only need to establish the reverse inequality: 
    \begin{equation} \label{eq:secondineqvt}
        v_t(r) \leq \sup_{\beta \in [t, 1]} B_t(\beta, r), \quad r > 0.
    \end{equation}
    Fix $r > 0$. 
    For convenience, define
    \begin{equation} \label{def:Jr}
        J_r(\rho) = L_0(\rho) - t \log (\rho - r), \quad \rho > r. 
    \end{equation}
    Observe that $J_r(\rho) \to \infty$ as $\rho \downarrow r$. 
    In addition, since 
    \[ J_r(\rho) \geq \log \rho - t \log \rho - t \log (1 - r/\rho) = (1-t) \log \rho - t \log (1 - r/\rho)  \]
    by Proposition \ref{prop:L0}, it follows that $J_r(\rho) \to \infty$ as $\rho \to \infty$. 
    Since $J_r$ is continuous, it attains its minimum at some $\rho_\ast > r$. 

    Recall the definition of $F_0$ in \eqref{def:F0}:
    \[ F_0(\rho) = \Prob(|X_1| \leq \rho), \quad \rho \geq 0. \]
    We also use 
    \[ F_0(\rho-) = \Prob(|X_1| < \rho), \quad \rho > 0, \]
    and we introduce the notation
    \begin{align*}
        L'_{0, -}(\rho) &= \lim_{h \downarrow 0} \frac{L_0(\rho) - L_0(\rho - h)}{h}, \\
        L'_{0, +}(\rho) &= \lim_{h \downarrow 0} \frac{L_0(\rho + h) - L_0(\rho)}{h}
    \end{align*}
    for the left and right derivatives of $L_0$, provided the limits exist.
    For the left derivative, using the representation in Proposition \ref{prop:L0}, 
    \[ \frac{L_0(\rho) - L_0(\rho-h)}{h} = \E \left[ \frac{\log \max\{\rho, |X_1|\} - \log \max \{\rho - h, |X_1|\}}{h} \right]. \]
    For fixed $|X_1| = x_1$, the expression inside the expectation converges to 
    \[ f(x_1) = \begin{cases}
        \frac{1}{\rho} & \text{ if } x_1 < \rho, \\
        0 & \text{ if } x_1 \geq \rho.
    \end{cases} \]
    In addition, for $h$ sufficiently small, 
    \[ 0 \leq \frac{\log \max \{\rho, x_1\} - \log \max \{\rho - h, x_1\}}{h} \leq \frac{\log \rho - \log (\rho - h)}{h} \leq \frac{2}{\rho}. \]
    Thus, the dominated convergence theorem applies, and we obtain
    \begin{equation} \label{eq:leftder}
        L'_{0, -}(\rho) = \E \left[f(|X_1|)\right] = \frac{F_0(\rho-)}{\rho}. 
    \end{equation}
    By similar reasoning, 
    \begin{equation} \label{eq:rightder}
        L'_{0, +}(\rho) = \frac{F_0(\rho)}{\rho}. 
    \end{equation}

    Recall the definition of $J_r$ given in \eqref{def:Jr}, and note that $J_r$ is a function of $L_0$.  
    Thus, we can now easily compute the left and right derivatives of $J_r$, which we denote as $J'_{r, -}$ and $J'_{r, +}$, respectively; in particular, the left and right derivatives of $J_r$ exist since the left and right derivatives of $L_0$ exist and $\rho \mapsto \log (\rho - r), \rho > r$ is differentiable.  
    Since $\rho_\ast$ is the global minimizer of $J_r$, we see that
    \begin{equation*}
        J'_{r, -}(\rho_\ast) \leq 0 \leq J'_{r,+}(\rho_\ast). 
    \end{equation*}
    Thus, from \eqref{eq:leftder} and \eqref{eq:rightder}, 
    \[ J'_{r,-}(\rho_\ast) = \frac{F_0(\rho_\ast-)}{\rho_\ast} - \frac{t}{\rho_\ast - r} \leq 0 \]
    and 
    \[ J'_{r,+}(\rho_\ast) = \frac{F_0(\rho_\ast)}{\rho_\ast} - \frac{t}{\rho_\ast - r} \geq 0. \]
    It follows that 
    \[ F_0(\rho_\ast-) \leq \frac{t \rho_\ast}{\rho_\ast - r} \leq F_0(\rho_\ast). \]
    Set 
    \begin{equation} \label{eq:defbetaast}
        \beta_\ast = \frac{t \rho_\ast}{\rho_\ast - r} > t.
    \end{equation}
    If $u < \beta_\ast$, then $u < F_0(\rho_\ast)$, and hence $Q_0(u) \leq \rho_\ast$. 
    Similarly, if $u > \beta_\ast$, then $u > F_0(\rho_\ast-)$ and $Q_0(u) \geq \rho_\ast$. 
    This implies that 
    \[ L_0(\rho_\ast) = \int_{0}^{\beta_\ast} \log \rho_\ast \d u + \int_{\beta_\ast}^1 \log Q_0(u) \d u = \beta_\ast \log \rho_\ast + \int_{\beta_\ast}^1 \log Q_0(u) \d u. \]
    Thus, by the choice of $\beta_\ast$ in \eqref{eq:defbetaast}, we see that
    \[ \rho_\ast = \frac{r \beta_\ast}{\beta_\ast - t} \quad \text{and} \quad \rho_\ast - r = \frac{rt}{\beta_\ast - t}, \]
    which allows us to arrive at 
    \begin{align*}
        v_t(r) &= J_r(\rho_\ast) \\
        &= \beta_\ast \log \rho_\ast + \int_{\beta_\ast}^1 \log Q_0(u) \d u - t \log (\rho_\ast - r) \\
        &= B_t(\beta_\ast, r). 
    \end{align*}
    This means
    \[ v_t(r) = B_t(\beta_\ast, r) \leq \sup_{\beta \in [t, 1]} B_t(\beta, r), \]
    which is precisely \eqref{eq:secondineqvt}.
    Combining \eqref{eq:firstineqvt} and \eqref{eq:secondineqvt} completes the proof. 
\end{proof}

With Lemma \ref{lem:vt-bt} in hand, we can now complete the proof of Lemma \ref{lem:vt}. 

\begin{proof}[Proof of Lemma \ref{lem:vt}]
    Let $f_t$ be given as in \eqref{def:ft}. 
    It follows from Proposition \ref{prop:quantile} that $f_t$ is non-decreasing.
    For $t \leq \beta \leq 1$ and $x \in \R$, define
    \[ D_t(\beta, x) = (\beta - t) x + \int_\beta^1 \log f_t(u) \d u. \]
    It follows that 
    \[ D_t(\beta, x) = \int_t^\beta x\d u + \int_\beta^1 \log f_t(u) \d u \leq \int_t^1 \max \{ x, \log f_t(u) \} \d u \]
    for any $t \leq \beta \leq 1$ and any $x \in \R$. 
    
    We claim that 
    \begin{equation} \label{eq:supDt}
        \sup_{\beta \in [t, 1]} D_t(\beta, x) = \int_t^1 \max\{ x, \log f_t(u) \} \d u 
    \end{equation}
    for any $x \in \R$. 
    Indeed, fix $x \in \R$. 
    Since $f_t$ is non-decreasing, suppose there exists $\beta_x$ with $t < \beta_x < 1$ so that $\log f_t(u) \leq x$ for all $u < \beta_x$ and $\log f_t(u) \geq x$ for all $u > \beta_x$.
    In this case, we see that \eqref{eq:supDt} holds as the supremum is attained at $\beta = \beta_x$.  
    On the other hand, if there does not exist such a $\beta_x$, then either $\log f_t(u) \leq x$ for all $u$ or $\log f_t(u) \geq x$ for all $u$.  
    This implies \eqref{eq:supDt} holds since the supremum is attained at either $\beta = t$ or $\beta = 1$.
    Thus, we conclude that \eqref{eq:supDt} holds.

    A direct calculation shows that 
    \[ \int_\beta^1 \log \left( \frac{u-t}{u} \right) \d u = \beta \log \beta - (\beta-t) \log(\beta-t) + (1-t) \log(1-t). \]
    Thus, for any $r > 0$,  
    \[ D_t(\beta, \log r) = B_t(\beta, r) + t \log t + (1-t) \log (1-t), \]
    where $B_t$ is defined in \eqref{def:Bt}. 
    Taking a supremum in $\beta$ and applying \eqref{eq:vtridenti} and \eqref{eq:supDt} yields
    \begin{align*}
        v_t(r) + t \log t + (1-t) \log (1-t) &= \int_t^1 \max \{ \log r, \log f_t(u) \} \d u \\
        &= \int_t^1 \log \max \{r, f_t(u) \} \d u,
    \end{align*}
    and the proof is complete.
\end{proof}

\subsection{Proof of Lemma \ref{lem:qnpointwise}} \label{sec:qnpointwise}

We now turn to the proof of Lemma \ref{lem:qnpointwise}. 
We first develop a number of tools needed for the proof; the proof of Lemma \ref{lem:qnpointwise} appears at the end of the section. 

\begin{lemma} \label{lem:lln}
    For any compact set $K \subset \C$, 
    \begin{equation} \label{eq:llnbnd}
        \limsup_{n \to \infty} \sup_{w \in K} \frac{1}{n} \log |p_n(w)| \leq \sup_{w \in K} \int_\C \log |w - z| \d \mu_0(z) 
    \end{equation}
    almost surely. 
\end{lemma}
\begin{remark}
   While the function 
   \[ w \mapsto \int_\C \log |w-z| \d \mu_0(z) \]
   appearing on the right-hand side of \eqref{eq:llnbnd} never takes the value $+\infty$ due to assumption \eqref{eq:lm}, it can take the value $-\infty$. 
   However, by Fubini's theorem and the fact that $\log |\cdot|$ is locally integrable, it follows that the function is finite Lebesgue-almost everywhere. 
\end{remark}
\begin{proof}[Proof of Lemma \ref{lem:lln}]
    For each $T > 0$, define the truncated logarithm
    \[ \log_T r = \max \{\log r, -T \}, \quad r \geq 0. \]
    For now, fix $T > 0$. 
    Since the function $\log_T |\cdot|$ is $e^T$-Lipschitz continuous on $[0, \infty)$, we see that 
    \begin{equation} \label{eq:ftlipschitz}
         \left|\log_T|z - w| - \log_T |z - w'|\right| \leq e^T |w - w'|
    \end{equation}
    for any $w, w', z \in \C$. 
    Choose $r > 0$ so that $K \subset \{z \in \C : |z| \leq r\}$. 
    Then $|\log_T|w-z|| \leq T + \log (1 + r + |z|)$ for any $w \in K$ and $z \in \C$. 
    By assumption \eqref{eq:lm}, this implies  
    \[ \E\left[\sup_{w \in K} \left| \log_T|w - X_1| \right|\right] < \infty. \]
    
    Let $\eps > 0$. 
    Since $K$ is compact, take $w_1, \ldots, w_N \in K$ to be a finite $\eps$-net of $K$, where $N$ is fixed, depending only on $\eps$ and $K$.  
    By the strong law of large numbers and the union bound, 
    \begin{equation} \label{eq:ftslln}
        \max_{l \in [N]} \left| \frac{1}{n} \sum_{j=1}^n \log_T|w_l - X_j| - \E[\log_T|w_l - X_1|] \right| \longrightarrow 0
    \end{equation}
    almost surely as $n \to \infty$.
    Using \eqref{eq:ftlipschitz}, we obtain
    \begin{align*}
        \sup_{w \in K} &\left| \frac{1}{n} \sum_{j=1}^n \log_T|w - X_j| - \E \left[ \log_T|w - X_1|\right] \right| \\
        &\qquad\qquad\leq \max_{l \in [N]} \left| \frac{1}{n} \sum_{j=1}^n \log_T|w_l - X_j| - \E \left[ \log_T|w_l - X_1|\right]\right| + 2 e^T \eps.  
    \end{align*}
    In view of \eqref{eq:ftslln} and the fact that $\eps > 0$ is arbitrary, we conclude that
    \[ \sup_{w \in K} \left| \frac{1}{n} \sum_{j=1}^n \log_T|w- X_j| - \E \left[ \log_T|w - X_1|\right] \right|  \longrightarrow 0 \]
    almost surely as $n \to \infty$. 
    Since 
    \begin{equation} \label{eq:logftineq}
        \log |w - X_1| \leq \log_T|w - X_1| 
    \end{equation}
    for any $w \in \C$, we conclude that
    \[ \limsup_{n \to \infty} \sup_{w \in K} \frac{1}{n} \log |p_n(w)| \leq \sup_{w \in K} \E \left[ \log_T|w - X_1|     \right]. \]
    In order to complete the proof, it only remains to show 
    \begin{equation} \label{eq:supwKdown}
        \sup_{w \in K} \E \left[ \log_T|w - X_1| \right] \downarrow \sup_{w \in K} \E \left[ \log |w-X_1| \right]
    \end{equation}
    as $T \to \infty$. 
    In fact, by \eqref{eq:logftineq}, we see that 
    \[ \sup_{w \in K} \E \left[ \log|w - X_1| \right] \leq \liminf_{T \to \infty} \sup_{w \in K} \E \left[ \log_T |w-X_1| \right], \]
    hence, it suffices to show
    \begin{equation} \label{eq:logtlimsup}
        \sup_{w \in K} \E \left[ \log|w - X_1| \right] \geq \limsup_{T \to \infty} \sup_{w \in K} \E \left[ \log_T |w-X_1| \right]. 
    \end{equation}

    To start, we claim that, for each $w \in K$, we have the pointwise convergence
    \begin{equation} \label{eq:logTptwise}
        \E \left[ \log_T|w - X_1| \right] \downarrow \E \left[ \log |w-X_1| \right] \in [-\infty, \infty) 
    \end{equation}
    as $T \to \infty$.
    Indeed, this can be deduced by considering the positive and negative parts of the random variable $Y_w = \log|w-X_1|$. 
    Its positive part $Y_{w,+}$ satisfies $Y_{w,+} \leq \log (1 + r + |X_1|)$.
    This implies the positive part has finite expectation; the negative part $Y_{w,-}$ may have infinite expectation,  but $\E[Y_w] = \E[Y_{w,+}] - \E[Y_{w,-}] \in [-\infty, \infty)$ is still well-defined.
    We express the truncated logarithm as
    \[ \log_T|w - X_1| = Y_{w,+} - \min\{Y_{w,-}, T\}. \]
    Since $\min\{Y_{w,-}, T\} \uparrow Y_{w,-}$ as $T \to \infty$, \eqref{eq:logTptwise} follows from the monotone convergence theorem. 

     Let $(T_l)_{l =1}^\infty$ be a positive sequence of real numbers tending to infinity, and choose $w_l \in K$ so that 
     \begin{equation} \label{eq:logsupbnd}
        \E[ \log_{T_l} |w_l - X_1|] \geq \sup_{w \in K} \E[\log_{T_l} |w - X_1|] - \frac{1}{l}. 
     \end{equation}
     By compactness of $K$, there exists a convergent subsequence of $(w_l)$; for notational simplicity, we denote the subsequence as $(w_l)$, and so $w_l \to w_\infty$ for some $w_\infty \in K$.  
     For any fixed $T > 0$, 
     \[ \E[\log_{T_l}|w_l - X_1|] \leq \E [\log_T |w_l - X_1|] \]
     whenever $T_l \geq T$. 
     Since $w \mapsto \E [\log_T |w - X_1|]$ is continuous, we obtain
     \[ \limsup_{l \to \infty} \E [\log_{T_l} |w_l - X_1|] \leq \E [\log_T |w_\infty - X_1|]. \]
     Since this is true for any $T > 0$, we can use the pointwise convergence of \eqref{eq:logTptwise} to conclude that 
     \begin{align*}
         \limsup_{l \to \infty} \sup_{w \in K} \E [\log_{T_l} |w - X_1|] &\leq \limsup_{l \to \infty} \E [\log_{T_l} |w_l - X_1|] \\
         &\leq \E [\log |w_\infty - X_1|] \\
         &\leq \sup_{w \in K} \E [\log |w - X_1|], 
     \end{align*}
     where the first inequality follows from \eqref{eq:logsupbnd}.
     This yields \eqref{eq:logtlimsup}, and the proof is complete.
\end{proof}

Recall the definition of $q_n$ in \eqref{def:qn}, and fix $r > 0$. 
Take $\rho > 0$, and note that Cauchy's integral formula implies 
\[ |q_n(r)| \leq \frac{1}{\rho^{k_n}} \max_{w \in \C: |w - r| = \rho} |p_n(w)| \]
for any $\rho > 0$. 
By Proposition \ref{prop:L0}, $L_0$ is non-decreasing, so
\[ \sup_{w \in \C: |w - r| = \rho} L_0(|w|) = L_0(r + \rho). \]
Applying Lemma \ref{lem:lln} with $K = \{w \in \C : |w-r| = \rho\}$, we find
\begin{align*}
    \limsup_{n \to \infty} \frac{1}{n} \log |q_n(r)| \leq L_0(r + \rho) - t \log \rho 
\end{align*}
almost surely. 
By the definition of $v_t$ in \eqref{def:vt} as the infimum, we can take a sequence $(\rho_l)$ tending to this infimum and work on the intersection of these probability one events to conclude that 
\begin{equation} \label{eq:limsupupper}
    \limsup_{n \to \infty} \frac{1}{n} \log |q_n(r)| \leq v_t(r) 
\end{equation}
almost surely.

In order to establish Lemma \ref{lem:qnpointwise}, it remains to establish the analogous convergence statement for $\liminf_{n \to \infty} \frac{1}{n}\log |q_n(r)|$. 
We now introduce the tools we will need. 

A \emph{multi-affine polynomial} is a multivariate polynomial that is an affine transformation in each of its variables separately when all other variables are held constant. 

\begin{lemma} \label{lem:lower-bound}
    Consider a nonzero multi-affine polynomial $f$ given by 
    \[ f(u_1, \ldots, u_N) = \sum_{S \subset [N]} a_S u_S \]
    for some complex coefficients $a_S \in \C, S \subset [N]$, where 
    \[ u_S = \prod_{j \in S} u_j. \]
    Recalling that $U_1, \ldots, U_N$ are independent random variables uniformly distributed on the unit circle centered at the origin in the complex plane, we have
    \begin{equation} \label{eq:expectationf}
        \E[ \log |f(U_1, \ldots, U_N)|] \geq \log \max_{S \subset [N]} |a_S| 
    \end{equation}
    and 
    \begin{equation} \label{eq:logvarbnd}
        \Var (\log |f(U_1, \ldots, U_N)|) \leq C N, 
    \end{equation}
    where $C > 0$ is an absolute constant. 
    In particular, for any nonzero coefficient $a_S$ and every $\eps > 0$, 
    \begin{equation} \label{eq:chebyshev-lower-bound}
        \Prob( \log |f(U_1, \ldots, U_N)| < \log |a_S| - \eps N) \leq \frac{C}{\eps^2 N}. 
    \end{equation}
\end{lemma}
\begin{proof}
    We consider $\T$, the unit circle centered at the origin in the complex plane with the normalized Haar measure (i.e., the uniform probability measure). 
    For $a, b \in \C$, not both zero, 
    \[ \log |a + b U_1| \eqd \log \max\{|a|, |b| \} + \log |1 + c U_1|, \]
    where $0 \leq c \leq 1$. 
    For $0 \leq c < 1$, we can write $\log |1 + ce^{i \theta}|$ as the real part of the principal branch of the complex logarithm to obtain 
    \[ \log |1 + c e^{i \theta}| = \sum_{l=1}^\infty (-1)^{l+1} \frac{c^l \cos(l \theta)}{l} \]
    for $\theta \in [0, 2\pi)$.
    Thus, 
    \[ \E [ \log |1 + cU_1|] = \frac{1}{2 \pi} \int_0^{2\pi} \log |1 + c e^{i\theta}|\d \theta = 0, \]
    and by orthogonality 
    \begin{equation*}
        \E [ \log^2 |1 + c U_1|] = \frac{1}{2} \sum_{l=1}^\infty \frac{c^{2l}}{l^2} \leq \frac{1}{2} \sum_{l=1}^\infty \frac{1}{l^2} = \frac{\pi^2}{12}. 
    \end{equation*}
    The result can also be established at $c = 1$ using a standard $L^2$-approximation argument: 
    \[ \left\| \log |1 + c \cdot| - \log |1 + \cdot | \right\|^2_{L^2(\T)} = \frac{1}{2} \sum_{l=1}^\infty \frac{(1-c^l)^2}{l^2} \longrightarrow 0 \]
    as $c \uparrow 1$. 
    We conclude that 
    \begin{equation} \label{eq:logU_1}
        \E[ \log |a + b U_1|] = \log \max \{|a|, |b|\}
    \end{equation}
    and 
    \begin{equation} \label{eq:log2U_1}
        \E[ \log^2 |a + b U_1|] \leq \log^2 \max \{|a|, |b|\} + \frac{\pi^2}{12}.
    \end{equation}

    We now proceed by induction in $N$ to show the zero set of a nonzero multi-affine polynomial on $\T^N$ has Haar measure zero and its logarithm is square integrable (i.e., its logarithm is an element of  $L^2(\T^N)$, where we view $\T^N$ with the product Haar measure). 
    For $N = 0$ the result is trivial, and the $N = 1$ case follows from \eqref{eq:log2U_1} and the calculations above.
    Assume the statements hold true for all nonzero multi-affine polynomials in $N-1$ variables and decompose 
    \[ f(u_1, \ldots, u_N) = f_1(u_1, \ldots, u_{N-1}) + u_N f_2(u_1, \ldots, u_{N-1}), \]
    where 
    \[ f_1(u_1, \ldots, u_{N-1}) = \sum_{S \subset [N-1]} a_S u_S, \quad f_2(u_1, \ldots, u_{N-1}) = \sum_{S \subset [N-1]} a_{S \cup \{N\}} u_S. \]
    At least one of $f_1, f_2$ is a nonzero polynomial. 
    Choose one of $f_1, f_2$ that is not identically zero and call it $g$. 
    By the induction hypothesis, 
    \[ \log |g(\cdot)| \in L^2(\T^{N-1}) \]
    and $g \neq 0$ almost everywhere.  
    It follows that 
    \[ m(u_1, \ldots, u_{N-1}) = \max \{|f_1(u_1, \ldots, u_{N-1})|, |f_2(u_1, \ldots, u_{N-1})| \} > 0 \]
    almost everywhere, where we take the equality to be the definition of $m$. 
    In addition, by the induction hypothesis, at least one of the nonzero polynomials $f_1, f_2$ is nonzero almost everywhere. 
    Since the one-variable linear polynomial 
    \[ u \mapsto f_1(u_1, \ldots, u_{N-1}) + u f_2(u_1, \ldots, u_{N-1}) \]
    has at most one zero on $\T$ for almost every choice of $(u_1, \ldots, u_{N-1}) \in \T^{N-1}$, it follows from Fubini's theorem that the zero set of $f$ in $\T^N$ has product Haar measure zero. 

    For notational simplicity, set $V = (U_1, \ldots, U_{N-1})$.
    Conditional on $V$, $f$ has the form
    \[ f(V, U_N) = f_1(V) + f_2(V) U_N, \]
    and so by \eqref{eq:log2U_1}
    \begin{equation} \label{eq:log2U_N}
        \E\left[ \log^2 |f(V, U_N)| \middle| V \right] \leq \log^2 |m(V)| + \frac{\pi^2}{12}. 
    \end{equation}
    It remains to show $\log m(\cdot) \in L^2(\T^{N-1})$.
    Since $f_1$ and $f_2$ are continuous on the compact set $\T^{N-1}$, there exists a constant $C' > 1$ so that $\log_+ m(v) \leq \log_+ C'$ for all $v \in \T^{N-1}$.
    Since $r \mapsto \log_- r = \max\{-\log r, 0\}$ is decreasing, we find
    \[ \log_- m(v) \leq \log_- |g(v)|, \quad v \in \T^{N-1}. \]
    The right-hand side is square integrable by the induction hypothesis, and we conclude that $\log m(\cdot) \in L^2(\T^{N-1})$.
    Returning to \eqref{eq:log2U_N}, we find
    \[ \E[ \log^2 |f(U_1, \ldots, U_N)|] < \infty. \]

    We now proceed by a second induction in $N$ to prove \eqref{eq:expectationf}.
    The base case follows from \eqref{eq:logU_1}. 
    Assume \eqref{eq:expectationf} holds for all nonzero multi-affine polynomials in $N-1$ variables. 
    Let $S \subset [N]$ be such that $a_S \neq 0$.
    We again decompose 
    \[ f(V, U_N) = f_1(V) + U_N f_2(V). \]
    By \eqref{eq:logU_1}, 
    \[ \E\left[ \log |f(V, U_N)| \middle| V \right] = \log \max \{|f_1(V)|, |f_2(V)|\}. \]
    If $N \notin S$, then $a_S$ is a coefficient of $f_1$ and 
    \begin{align*}
        \E [\log |f(V, U_N)|] &= \E \log \max \{|f_1(V)|, |f_2(V)|\} \\
        &\geq \E [\log |f_1(V)|] \\
        &\geq \log |a_S|,
    \end{align*}
    where the last inequality follows from the induction hypothesis applied to $f_1$.
    If $N \in S$, then $a_S$ appears as a coefficient in $f_2$ and a similar argument gives 
    \begin{align*}
        \E[ \log |f(V, U_N)|] \geq \E [\log |f_2(V)|] \geq \log |a_S| 
    \end{align*}
    by the induction hypothesis applied to $f_2$. 
    Since the result is true for any $S \subset [N]$ with $a_S \neq 0$, we conclude that
    \[ \E [\log |f(V, U_N)|] \geq \log \max_{S \subset [N]} |a_S|, \]
    as desired. 

    We now prove \eqref{eq:logvarbnd}.
    For each $j \in [N]$, we define $U^{(j)} = (U_1, \ldots, U_{j-1}, U_{j+1}, \ldots, U_N)$ to contain all of the independent uniform random variables, except for the $j$-th.
    Fixing $U^{(j)}$, the multi-affine assumption implies
    \[ f(U_1, \ldots, U_N) = f_{1,j}(U^{(j)}) + U_j f_{2, j}(U^{(j)}). \]
    Almost surely the two coefficients $f_{1,j}(U^{(j)}), f_{2, j}(U^{(j)})$ are not both zero by the induction proof previously given.
    Thus, we obtain the one-variable variance bounds
    \[ \Var\left(\log |f(U_1, \ldots, U_N)| \middle| U^{(j)} \right) \leq \frac{\pi^2}{12}, \quad j \in [N] \]
    by applying \eqref{eq:logU_1} and \eqref{eq:log2U_1}. 
    Therefore, by the Efron--Stein inequality (see Theorem 3.1 in \cite{MR3185193}), we conclude that 
    \[ \Var \left( \log |f(U_1, \ldots, U_N)| \right) \leq N \frac{\pi^2}{12}, \]
    as required. 

    Finally, \eqref{eq:chebyshev-lower-bound} follows from \eqref{eq:expectationf}, \eqref{eq:logvarbnd}, and an application of Chebyshev's inequality. 
\end{proof}

We say $\beta \in (t, 1]$ is \emph{good} if $Q_0(\beta) > 0$ or if $\beta = 1$.
We let $G_t$ be the set of good values: 
\[ G_t = \{ \beta \in (t, 1) : Q_0(\beta) > 0 \} \cup \{1\}. \]
We include $\beta = 1$ to ensure the set is nonempty. 
We provide the following refinement of Lemma \ref{lem:vt-bt}.
\begin{lemma} \label{lem:Gt}
    Letting $p = \mu_0(\{0\})$, we have 
    \[ G_t = \begin{cases}
        (\max\{p,t\}, 1] & \text{ if } p < 1, \\
        \{1\} & \text{ if } p = 1. 
    \end{cases}\]
    In addition, for any $r > 0$, 
    \begin{equation*} 
        v_t(r) = \sup_{\beta \in G_t} B_t(\beta, r), 
    \end{equation*}
    where $B_t$ is defined in \eqref{def:Bt}. 
\end{lemma}
\begin{proof}
    In view of Lemma \ref{lem:vt-bt}, after identifying the set $G_t$, it will suffice to show 
    \begin{equation} \label{eq:vt-Gt}
        \sup_{\beta \in [t, 1]} B_t(\beta, r) = \sup_{\beta \in G_t} B_t(\beta, r). 
    \end{equation}
    We note that $B_t(1, r)$ is finite, and hence neither supremum could be $-\infty$. 

    We claim that $Q_0(u) = 0$ if and only if $u \leq p$. 
    Indeed, if $u \leq p$, then $F_0(0) = p \geq u$, so $Q_0(u) = 0$.
    Conversely, if $u > p$, right continuity of $F_0$ at zero implies there exists $\delta > 0$ so that $F_0(\delta) < u$. 
    In other words, $Q_0(u) \geq \delta > 0$. 
    In particular, this means the set $G_t$ is either of the form $(\max\{p, t\}, 1]$ or $\{1\}$, where the latter case only arises when $p = 1$.

    Fix $r > 0$, and consider $B_t(\beta, r)$. 
    If $t \leq \beta < p$, it follows that $\log Q_0(u) = -\infty$ for all $u \in (\beta, p)$, and so $B_t(\beta, r) = -\infty$.
    Such a value of $\beta$ cannot contribute to the supremum.  

    All $\beta > \max\{p, t\}$ are good.  
    Therefore, if $1 > p \geq t$, then $B_t(p, r)$ can be approximated by good values: $B_t(p, r) = \lim_{\beta \downarrow p} B_t(\beta, r)$. 
    Similarly, the endpoint $\beta = t$ can also be approximated by good values when $p < t$: $B_t(t, r) = \lim_{\beta \downarrow t}B_t(\beta, r)$. 
    We conclude that \eqref{eq:vt-Gt} holds, and the proof is complete. 
\end{proof}

Recall that $R_1, \ldots R_n$ are the radii of $X_1, \ldots, X_n$. 
Write the order statistics of the random sample $R_1, \ldots, R_n$ as 
\[ R_{n,(1)} \leq \cdots \leq R_{n,(n)}. \]

\begin{lemma} \label{lem:order-stats}
Fix $\beta \in (0, 1)$ so that $Q_0(\beta) > 0$. 
For any sequence of positive integers $(\alpha_n)$ satisfying 
\[ \frac{\alpha_n}{n} \longrightarrow 1 - \beta, \]
we have
\[ \frac{1}{n} \sum_{j=n-\alpha_n+1}^n \log R_{n,(j)} \longrightarrow \int_\beta^1 \log Q_0(u) \d u \]
almost surely as $n \to \infty$. 
\end{lemma}
\begin{proof}
    Recall the definition of the CDF $F_0$ of $R_1$ in \eqref{def:F0}. 
    Choose $0 < a < Q_0(\beta)$. 
    This implies $F_0(a) < \beta$, and hence by the strong law of large numbers
    \[ \frac{1}{n} |\{j \in [n] : R_j \leq a\}| \longrightarrow F_0(a) < \beta \]
    almost surely as $n \to \infty$. 
    Since $\frac{n-\alpha_n}{n} \to \beta$, this implies that almost surely for $n$ sufficiently large, 
    \begin{equation} \label{eq:Rnbig}
        R_{n, (n - \alpha_n + 1)} > a. 
    \end{equation}
    In other words, almost surely for $n$ sufficiently large, the $\alpha_n$ largest radii exceed $a$.

    Define the truncated random variables 
    \[ Y_j = \max \{\log R_j, \log a \}. \]
    Note that $Y_1, Y_2, \ldots$ are i.i.d. random variables that are bounded below and are integrable due to assumption \eqref{eq:lm}. 

    Let $\nu_n$ be the empirical measure of $Y_1, \ldots, Y_n$:
    \[ \nu_n = \frac{1}{n} \sum_{j=1}^n \delta_{Y_j}. \]
    We let $\nu$ be the law of $Y_1$ so that 
    \[ \nu_n \longrightarrow \nu \]
    weakly almost surely as $n \to \infty$ by the Glivenko--Cantelli theorem. 
    By another application of the strong law of large numbers, 
    \[ \int_\R |y| \d \nu_n(y) = \frac{1}{n} \sum_{j=1}^n |Y_j| \longrightarrow \E [|Y_1|] = \int_\R |y| \d \nu(y) \]
    almost surely. 
    By the standard characterization of the Wasserstein $L_1$ metric $W_1$ (see \cite{MR1964483} for a standard reference as well as \cite{MR4028181} for an overview and additional references), it follows that $W_1(\nu_n, \nu) \to 0$ almost surely as $n \to \infty$. 
    In one dimension, Proposition 1 from \cite{MR3608466} implies that
    \begin{equation} \label{eq:W1}
        \int_0^1 |Q_{n}(u) - Q_{Y_1}(u)| \d u = W_1(\nu_n, \nu) \longrightarrow 0 
    \end{equation}
    almost surely, where $Q_n$ is the quantile function of $\nu_n$ and $Q_{Y_1}$ is the quantile function of $Y_1$ (equivalently, $Q_{Y_1}$ is the quantile function of the distribution $\nu$).

    In view of Proposition \ref{prop:quantile}, for $u \geq \beta$, 
    \[ Q_0(u) \geq Q_0(\beta) > a, \]
    and hence by a direct calculation of the distribution of $Y_1$
    \begin{equation} \label{eq:QY1Q0}
        Q_{Y_1}(u) = \log Q_0(u), \quad u \in [\beta, 1]. 
    \end{equation}
    Let 
    \[ Y_{n, (1)} \leq \cdots \leq Y_{n, (n)} \]
    be the order statistics of the sample $Y_1, \ldots, Y_n$. 
    It follows that
    \begin{equation} \label{eq:YjQn}
        \frac{1}{n} \sum_{j=n - \alpha_n + 1}^n Y_{n, (j)} = \int_{\beta_n}^1 Q_{n}(u) \d u, 
    \end{equation}
    where $\beta_n = \frac{n - \alpha_n}{n}$. 
    We also have 
    \begin{align}  \label{eq:betanQconv}
        \left| \int_{\beta_n}^1 Q_{n}(u) \d u - \int_\beta^1 Q_{Y_1}(u) \d u \right| \leq \int_0^1 | Q_{n}(u) - Q_{Y_1}(u) | \d u + \left| \int_{\beta_n}^\beta Q_{Y_1}(u) \d u \right|.
    \end{align}
    The first term on the right-hand side converges almost surely to zero by \eqref{eq:W1}.
    We note that the second term will converge to zero since $Q_{Y_1}$ is integrable on $(0, 1)$. 
     Indeed, if $\xi_0$ is a random variable uniformly distributed on $(0, 1)$, then $Q_{Y_1}(\xi_0)$ has the same distribution as $Y_1$ (see, for example, Proposition 2 in \cite{MR3072795}), and so 
    \[ \int_0^1 |Q_{Y_1}(u) | \d u = \E[|Q_{Y_1}(\xi_0)|] = \E[ |Y_1|] < \infty. \]
    Thus, the second term on the right-hand side of \eqref{eq:betanQconv} converges to zero since $\beta_n \to \beta$. 
    In view of \eqref{eq:YjQn}, we conclude that 
    \[ \frac{1}{n} \sum_{j=n - \alpha_n + 1}^n Y_{n, (j)} \longrightarrow \int_\beta^1 Q_{Y_1}(u) \d u = \int_\beta^1 \log Q_0(u) \d u \]
    almost surely as $n \to \infty$, where the equality follows from \eqref{eq:QY1Q0}.
    By \eqref{eq:Rnbig}, almost surely for $n$ sufficiently large
    \[ \frac{1}{n} \sum_{j=n - \alpha_n + 1}^n Y_{n, (j)}  = \frac{1}{n} \sum_{j=n - \alpha_n + 1}^n \log R_{n, (j)}, \]
    and the proof is complete. 
\end{proof}

We can now complete the proof of Lemma \ref{lem:qnpointwise}. 
\begin{proof}[Proof of Lemma \ref{lem:qnpointwise}]
    By our choice of normalization, 
    \[ q_n(r) = e_{m_n}(r - X_1, \ldots, r - X_n) = e_{m_n}(r - R_1U_1, \ldots, r-R_nU_n), \]
    where $e_{m_n}$ is the elementary symmetric polynomial of degree $m_n$.
    Expanding out the products, we find
    \[ q_n(r) = \sum_{S \subset [n] : |S| \leq m_n} a_{n, S}(r) U_S, \]
    where 
    \[ U_S = \prod_{j \in S} U_j \]
    and $a_{n, S}(r)$ are the coefficients containing powers of $r$ and the radii $R_1, \ldots, R_n$.  
    Here, $|S|$ denotes the cardinality of the set $S$. 
    The coefficients can be written explicitly as 
    \[ a_{n, S}(r) = (-1)^{|S|} \binom{n - |S|}{m_n - |S|} r^{m_n - |S|} R_S, \quad S \subset [n], |S| \leq m_n, \]
    where 
    \[ R_S = \prod_{j \in S} R_j. \]
    To see this, we note that the coefficient $a_{n, S}$ is constructed by choosing $-R_jU_j$ for every $j \in S$ and then choosing $m_n - |S|$ copies of $r$ for the remaining variables. 

    Fix $\beta \in G_t$. For the moment, assume $\beta < 1$. 
    Choose a sequence $(\alpha_n)$ so that $\frac{\alpha_n}{n} \to 1- \beta$ as $n \to \infty$, and let $n$ be sufficiently large so that $\alpha_n \leq m_n$.
    Let $S_n$ be the set of indices corresponding to the $\alpha_n$ largest radii from the sample $R_1, \ldots, R_n$ (with ties broken arbitrarily).  
    Then $S_n$ is a random set; after we condition on $R_1, \ldots, R_n$ below, we will treat $S_n$ as deterministic.
    Stirling's formula gives
    \[ \frac{1}{n} \log \binom{n - \alpha_n}{m_n - \alpha_n} \longrightarrow \beta \log \beta - (\beta - t) \log (\beta - t) - t \log t \]
    as $n \to \infty$, and Lemma \ref{lem:order-stats} implies
    \[ \frac{1}{n} \sum_{j \in S_n} \log R_{j} = \frac{1}{n} \sum_{j=n-\alpha_n + 1}^n \log R_{n, (j)} \longrightarrow \int_\beta^1 \log Q_0(u) \d u \]
    almost surely. 
    Combining these results, we conclude that 
    \begin{equation} \label{eq:Btconv}
        \frac{1}{n} \log |a_{n, S_n}(r)| \longrightarrow B_t(\beta, r) 
    \end{equation}
    almost surely as $n \to \infty$, where $B_t$ is given in \eqref{def:Bt}. 
    When $\beta = 1$, we have (with $S_n = \emptyset$)
    \[ a_{n, \emptyset} = \binom{n}{m_n} r^{m_n}, \]
    and taking $\alpha_n = 0$, we again obtain \eqref{eq:Btconv}. 

    Let $\eps > 0$, and fix $r > 0$.
    Using Lemma \ref{lem:Gt}, choose $\beta \in G_t$ so that 
    \begin{equation} \label{eq:bt-vt}
        B_t(\beta, r) \geq v_t(r) - \eps. 
    \end{equation}
    Conditioning on the radii $R_1, R_2, \ldots$, Lemma \ref{lem:lower-bound} gives\footnote{The assumptions of Lemma \ref{lem:lower-bound} are satisfied since $q_n(r)$ is clearly multi-affine in $U_1, \ldots, U_n$, and when $r > 0$, its constant coefficient is nonzero. In addition, $a_{n, S_n}$ is nonzero almost surely for $n$ sufficiently large by \eqref{eq:Rnbig}. }, for any $\eps > 0$, 
    \[ \Prob_{U_1, \ldots, U_n}(\log |q_n(r)| < \log |a_{n, S_n}(r)| - \eps n ) \leq \frac{C}{\eps^2 n}, \]
    where $C > 0$ is an absolute constant.
    Combining this with \eqref{eq:Btconv} and \eqref{eq:bt-vt}, we arrive at 
    \[ \lim_{n \to \infty} \Prob \left( \frac{1}{n} \log |q_n(r)| \geq v_t(r) - 2 \eps \right) = 1. \]
    In view of \eqref{eq:limsupupper} and the fact that $\eps > 0$ was arbitrary, we conclude that 
    \[ \frac{1}{n} \log |q_n(r)| \longrightarrow v_t(r) \]
    in probability as $n \to \infty$.
    This completes the proof.
\end{proof}

\bibliographystyle{abbrv}
\bibliography{bib}

\end{document}